\documentclass[12pt,hyperref]{article}%字体大小
\usepackage{graphicx} %% Package for Figure
\usepackage{float} %% Package for Float
\usepackage{amssymb}
\usepackage{amsmath}
\usepackage{mathtools}
\usepackage[thmmarks,amsmath]{ntheorem} %% If amsmath is applied, then amsma is necessary
\usepackage{lineno}
\usepackage{algorithm}
\usepackage{algpseudocode}
\usepackage{amsmath}
\usepackage{bm} %% Bold Mathematical Symbols
\usepackage{extarrows}
\usepackage{mathrsfs} %% Swash letter

\usepackage{algorithm,algpseudocode,float}
\usepackage{lipsum}

\makeatletter
\newenvironment{breakablealgorithm}
  {% \begin{breakablealgorithm}
   \begin{center}
     \refstepcounter{algorithm}% New algorithm
     \hrule height.8pt depth0pt \kern2pt% \@fs@pre for \@fs@ruled
     \renewcommand{\caption}[2][\relax]{% Make a new \caption
       {\raggedright\textbf{\ALG@name~\thealgorithm} ##2\par}%
       \ifx\relax##1\relax % #1 is \relax
         \addcontentsline{loa}{algorithm}{\protect\numberline{\thealgorithm}##2}%
       \else % #1 is not \relax
         \addcontentsline{loa}{algorithm}{\protect\numberline{\thealgorithm}##1}%
       \fi
       \kern2pt\hrule\kern2pt
     }
  }{% \end{breakablealgorithm}
     \kern2pt\hrule\relax% \@fs@post for \@fs@ruled
   \end{center}
  }
\makeatother
\usepackage{booktabs} %% tables with three lines
{%% Environment of Proof
    \theoremstyle{nonumberplain}
    \theoremheaderfont{\bfseries}
    \theorembodyfont{\normalfont}
    \theoremsymbol{\mbox{$\Box$}}
    \newtheorem{proof}{Proof}
}

\usepackage{theorem}
\newtheorem{theorem}{Theorem}[section]

\newtheorem{lemma}{Lemma}[section]
\newtheorem{definition}{Definition}[section]

\newtheorem{fact}{Fact}[section]
\newtheorem{claim}{Claim}[section]
\newtheorem{subclaim}{Subclaim}%[subsection]

\newtheorem{problem}{Problem}[section]
\newtheorem{conjecture}{Conjecture}[section]
{%% Environment of Remark
    \theoremheaderfont{\bfseries}
    \theorembodyfont{\normalfont}
    
}
\newcommand{\RNum}[1]{\uppercase\expandafter{\romannumeral #1\relax}}

\graphicspath{{figure/}}                                 %% Path of Figures
\usepackage[a4paper]{geometry}                        %% Paper size
\UseRawInputEncoding
\begin{document}

\title{The generalizations of Hamiltonian in oriented graphs\thanks{The author's work is supported by NNSF of China (No.12071260,11671232)}}

\author{Jia Zhou, Zhilan Wang, Jin Yan\thanks{Corresponding author. E-mail adress: yanj@sdu.edu.cn}  \unskip\\[2mm]
School of Mathematics, Shandong University, Jinan 250100, China}

\date{}
\maketitle

%a prescribed cycle-factor with
%MS+++++++++++++++++++++ Abstract +++++++++++++++++++++++++
\begin{abstract}
%Given a digraph $D$, a prescribed cycle-factor in $D$ is a set of disjoint cycles with given lengths that covers all vertices of it.

An oriented graph is an orientation of a simple graph. In 2009, Keevash, K\"{u}hn and Osthus proved that every sufficiently large oriented graph $D$ of order $n$ with $(3n-4)/8$ is Hamiltonian. Later, Kelly, K\"{u}hn and Osthus showed that it is also pancyclic. Inspired by this, we show that for any given constant $t$ and positive integer partition $n = n_1 + \cdots + n_t$, if $D$ is an oriented graph on $n$ vertices with minimum semidegree at least $(3n-4)/8$, then it contains $t$ disjoint cycles of lengths $n_1,\ldots , n_t$. Also, we determine the bounds on the semidegree of sufficiently large oriented graphs  that are strongly Hamiltonian-connected, $k$-ordered Hamiltonian and spanning $k$-linked.

%Hamiltonian is one of the most central notions in graph theory. An oriented graph is an orientation of a simple graph. In this paper, we will discuss several problems related to the generalization of Hamiltonian in oriented graphs. In more detail, we show that for any given constant $t$ and positive integer partition $n = n_1 + \cdots + n_t$, if $D$ is an oriented graph on $n$ vertices with minimum semidegree at least $(3n-4)/8$, then it contains $t$ disjoint cycles of lengths $n_1,\ldots , n_t$. Furthermore, we determine some relevant bounds on the semidegree of sufficiently large oriented graphs  that are strongly Hamiltonian-connected, $k$-ordered Hamiltonian and spanning $k$-linked.
\end{abstract}
%we also consider the strongly Hamiltonian-connected problem, the $k$-ordered Hamiltonian problem and the $k$-linkage problem in oriented graphs.%In addition, every sufficient large oriented graph $D$ on $n$ vertices with $\delta^0(D)\geq 3n/8+9$, then the oriented graph $D^\prime$ obtained from $D$ by adding a vertex $x$ such that $d^0_D(x)\geq 3n/16$ is Hamiltonian.


\vspace{1ex}

\section{Introduction}
A digraph $D$ is \emph{Hamiltonian} if $D$ contains a cycle which encounters every vertex only once. Hamiltonian is one of the most central notions in graph theory, and has been intensively studied by numerous researchers in recent decades. One of the sufficient conditions for Hamiltonian in digraphs was established by Ghouila-Houri \cite{Ghouila(1960)}. He confirmed that any digraph $D$ on $n$ vertices with $\delta^0(D)\geq n/2$ is Hamiltonian, where \emph{the minimum semidegree} $\delta^0(D) = \min \{\delta^+(D), \delta^-(D) \}$. However, the situation is complicated when one considers Hamiltonian of \emph{oriented graphs}, which is a digraph that can be obtained from a (simple) undirected graph by orienting its edges. The question concerning the analogue of the existence of the Hamiltonian cycle in oriented graphs was raised by Thomassen \cite{Thomassen(1979)}. In 2009, Keevash, K\"{u}hn and Osthus \cite{Keevash(2009)} gave an exact answer to the question as follows.
\begin{theorem}\label{the1}
\cite{Keevash(2009)} There exists a number $n_0$ so that any oriented graph $D$ on $n\geq n_0$ vertices with $\delta^0(D)\geq (3n-4)/8$ is Hamiltonian.
\end{theorem}

Note that the semidegree condition of Theorem \ref{the1} is tight, some relevant counterexamples are indicated in \cite{Keevash(2009)}. Furthermore, Kelly, K\"{u}hn and Osthus \cite{Kelly(2009)} showed that the same condition as Theorem \ref{the1} guarantees that $D$ is \emph{pancyclic}, which contains a cycle of length $l$ for every $l\in \{3, 4,\ldots,n\}$.%Packing results have proved useful in the study of graph properties.  Many results fall into this field and one of the most typical result is   (), which by .

%In this paper,  we discuss several natural ways of strengthening the notion of Hamiltonian.
In this paper, we will discuss several natural problems related to the Hamiltonicity. Note that a \emph{cycle-factor} of a digraph is a set of disjoint cycles that covers all vertices of it. Hence, the Hamiltonian cycle can be regarded as a cycle-factor with only one cycle. We are particularly interested in problems concerning cycle-factors in oriented graphs. In 2009, Keevash and Sudakov \cite{Sudakov(2009)} stated that there exist constants $c, C > 0$ such that for sufficiently large oriented graph $D$ on $n$ vertices with minimum semidegree at least $(1/2-c)n$, if $n_1,\ldots ,n_t$ are numbers satisfying $\sum^t_{i=1} n_i \leq n- C$, then $D$ contains disjoint cycles of lengths $n_1,\ldots,n_t$. In this paper, we prove the following theorem.

\begin{theorem}\label{main1}
Given an integer $t\geq 1$, there exists number $n_0$, so that any oriented graph $D$ on $n\geq n_0$ vertices with $\delta^0(D)\geq (3n-4)/8$ contains $t$ disjoint cycles of lengths $n_1,\ldots , n_t$, where $n = n_1 +\cdots + n_t$ is any positive integer partition of $n$.
\end{theorem}

Theorem \ref{main1} states that there exists any cycle-factor with a limited cycle number instead of "almost" cycle-factor in oriented graphs with the semidegree of Theorem \ref{the1}. Note that the semidegree of Theorem \ref{main1} is sharp for cycle-factor problem in oriented graphs, since it is tight for the existence of Hamiltonian cycle in oriented graphs.

%Thus, it is natural to ask what condition does a digraph satisfy to guarantee the existence of a Hamiltonian cycle.

%Furthermore,  Recently, Wang, Yan and Zhang \cite{Wang(2022)} established some results on the cycle-factor, which refined the results by Keevash and Sudakov \cite{Sudakov(2009)}.

Also, Wang, Wang and Yan \cite{Wang(2023)} made a new progress for digraphs, they showed the following result. For every digraph $D$ with $W \subseteq V(D)$, if $|W|\geq 2k$ and $d_D(x) + d_D(y) \geq 3n - 3$ for all $\{ x, y\} \subseteq W$, then for any integer partition $|W| = \sum^k_{i=1} n_i$ with $n_i \geq 2$ for each $i$, there are $k$ disjoint cycles containing exactly $n_1,\ldots , n_k$ vertices of $W$, respectively. For the cycle-factor problems in graphs, El-Zahar \cite{El-Zahar(1984)} conjectured that one may get any cycle-factor with lengths $n_1,n_2,\ldots,n_t$ in graphs of order $n$ with the minimum degree at least $ \sum^t_{i=1}\lceil n_i/2\rceil$.
This fascinating conjecture has been verified for $t = 2$ in \cite{El-Zahar(1984)}. In particular, it has also been confirmed for sufficiently large graphs in \cite{Abbasi(1999)}. And more relevant results can be found in references \cite{Hajnal(1963), Wang(2010), Wang(2012)}.
%For this fascinating conjecture, it is verified for . In particular, this conjecture holds for sufficiently large graphs, which was shown in.

%El-Zahar \cite{El-Zahar(1984)} verified his conjecture only for $t = 2$. After more than ten years, Abbasi, in his Ph.D Thesis \cite{Abbasi(1999)} justified the conjecture for sufficiently large graphs. Specially, Corr\'{a}di-Hajnal's theorem  is a special case of $n_i = 3$ for all $i\in [t]$. In 2010, Wang \cite{} showed the case $n_i = 4$, for any $i\in [t]$. Also in 2012, Wang \cite{} also gave a minimum degree condition for the case $n_i = 5$, for any $i\in [t]$.

Another natural generalization of Hamiltonianity is to require the existence of a Hamiltonian path between any pair of vertices. We call $D$ is \emph{strongly Hamiltonian-connected} if for any two vertices $x$ and $y$, there is a Hamiltonian path from $x$ to $y$. In addition, Berge \cite{Berge(1985)} demonstrated that every digraph $D$ on $n$ vertices with $\delta^0(D)\geq (n+1)/2$ is strongly Hamiltonian-connected. Theorem \ref{cor2} yields a result about strongly Hamiltonian-connected for sufficiently large oriented graphs.
\begin{theorem}\label{cor2}
There is an integer $n_0 $ such that every oriented graph $D$ on $n\geq n_0$ vertices with $\delta^0(D)\geq 3n/8+2$ is strongly Hamiltonian-connected.
\end{theorem}

In addition, it also makes sense to seek a Hamiltonion cycle visiting several vertices in a specific order. A digraph $D$ is \emph{$k$-ordered Hamiltonian} if for every sequence $s_1,\ldots,s_k$ of distinct vertices of $D$ there is a directed Hamiltonian cycle which encounters $s_1,\ldots,s_k$ in this order. Also, we call a digraph $D$ is \emph{$k$-linked} if it contains a system of disjoint paths $P_1, P_2,\ldots,P_k$ such that $P_i$ is a path from $x_i$ to $y_i$, for every choice of distinct vertices $x_1, \ldots ,x_k, y_1, \ldots,y_k$. If the union of these $k$ disjoint paths span $D$, we call $D$ is \emph{spanning $k$-linked}. Indeed, every $k$-linked digraph is $k$-ordered and every $2k$-ordered digraph is $k$-linked. In 2008, K\"{u}hn, Osthus, and Young \cite{Kuhn(2008)} proved that every digraph $D$ on $n\geq n_0(k)$ vertices with $\delta^0(D) \geq \lceil(n + k)/2\rceil - 1$ is $k$-ordered Hamiltonian. Here, we give a result about $k$-ordered Hamiltonian and spanning $k$-linked in oriented graphs.
%is an immediate consequence of Theorem \ref{cor2}

\begin{theorem}\label{cor3}
For any positive integer $k$, there is an integer $n_0$ such that for every oriented graph $D$ on $n\geq n_0$ vertices, the following holds.\\
\indent $(i)$ If $\delta^0(D)\geq 3n/8+5k/2-2$, then $D$ is $k$-ordered Hamiltonian.\\
\indent $(ii)$ If $\delta^0(D)\geq 3n/8+7k/2-1$, then $D$ is spanning $k$-linked.
\end{theorem}

The rest of the paper is organized as follows. The aim of Section 2 is to prepare some notation and basic tools used in the paper. In Section 3, we show some lemmas which are useful in the proof of Theorem \ref{main1}. And then we give the proof of Theorem \ref{main1}. The proof of Theorem \ref{cor2} and Theorem \ref{cor3} are shown in Section 4 and Section 5, respectively. Section 6 mentions some related problems about the generalizations of Hamiltonianity in oriented graphs.

\section{Preliminaries and tools}

\subsection{Notation and definitions}
Let $D=(V,A)$ be a digraph with vertex set $V$ and arc set $A$. We write $|D|$ for the number of vertices in $D$ and denote $a(D)$ to be the number of arcs in $D$. Given two vertices $x, y \in V$, we write $xy$ for the arc directed from $x$ to $y$. For a vertex $x$ of $D$, we write $N^+_D (x) = \{y: xy\in A\}$ (resp. $N^-_D (x)=\{y: yx\in A\} $) to be the \emph{out-neighbourhood} (resp. \emph{in-neighbourhood}) and $d^+_D (x) = |N^+_D (x)|$ (resp. $d^-_D (x)=|N^-_D (x)|$) to be the \emph{out-degree} (resp. \emph{in-degree}) of $x$. %Similarly, let $$N^-_D (x) = \{y: yx\in A(D)\} \text{ and } d^-_D (x) = |N^-_D (x)|$$ be the \emph{in-neighborhood} and the \emph{in-degree} of $x$, respectively.
Further, let $\delta^+(D)$ and $\delta^-(D)$ denote \emph{the minimum out-degree} and \emph{the minimum in-degree} of $D$. Let \emph{minimum degree $ \delta(D)=\min \{d^+(v)+d^-(v): v\in V(D)\}$}. Given a set $X\subseteq V$, the subdigraph of $D$ induced by $X$ is denoted by $D[X]$ and let $D - X$ denote the digraph obtained from $D$ by deleting $X$ and all arcs incident with $X$. Denote $N^+_D(X)$ to be the union of the out-neighbourhood of every vertex $x\in X$. For two subsets $X,Y$ in $V$, denote $A(X, Y)$ to be the set of arcs from $X$ to $Y$, and define $a(X,Y)=|A(X,Y)|$. A matching in a digraph is a set of disjoint arcs with no common endvertices. A matching $M$ is maximum if $M$ contains the maximum possible number of disjoint arcs. And in this paper, the length of a path is its vertices. We abbreviate the term "vertex disjoint" as "disjoint".

Throughout this paper, the notation $0<\beta\ll\alpha$ is used to make clear that $\beta$ can be selected to be sufficiently small corresponding to $\alpha$ so that all calculations required in our proof are valid. For a positive integer $t$, simply write $\{1, \ldots , t\}$ as $[t]$. For convenience, write $a\pm b$ to be an unspecified real number in the interval $[a- b,a + b]$. Next, we mention the important definition as follows.

\begin{definition}\emph{(}robust $(\mu, \tau)$-outexpander\emph{)}\label{def1}
Given $0<\mu \leq \tau < 1$, we say a digraph $R$ is a robust $(\mu, \tau)$-outexpander if $|N^+_R(S)|\geq |S| + \mu |R|$ for every $S\subseteq V(R)$ satisfying $\tau |R| \leq |S| \leq (1-\tau)|R|$.
\end{definition}

%To better characterize the properties of the extremal oriented graphs, which is defined in Definition \ref{def4}, we give the following notation. Here and subsequently, s
In fact, our proof use the technique of \cite{Keevash(2009)}, which finds a Hamiltonian cycle in an oriented graph. Therefore, we also continue use the symbol of \cite{Keevash(2009)}. Set $c$ to be some constant, let $\mu$ be a sufficiently small real number, and write $\alpha=(1/100 - c\sqrt{\mu})n/4$. We partition an oriented graph $D$ on $n$ vertices into four parts $D_1, D_2, D_3, D_4$. For simplicity of notation, for a vertex $x$ in $D_i$, if the cardinality of $N^+(x)\cap D_j$ is at least $\alpha$, then briefly write $D_i:(D_j)^{>\alpha}$. Next, we give the more definitions.

%we let '$>\alpha$' stand for an intersection of size at least  and let '$< \alpha$' describes an intersection of size at most $(\alpha + 3000 \sqrt{\mu})n/4$. , $\alpha_2=(1/100 + c \sqrt{\mu})n/4$; Likewise, if the cardinality of $N^-(x)\cap D_j$ is at most $\alpha_2$, then briefly write $D_i:(D_j)_{<\alpha_2}$.

\begin{definition}\emph{(}A vertex is acceptable\emph{)}\label{def3}
\emph{If a vertex satisfies one of the following 16 properties, then it is acceptable.\\
 \indent $\bullet$ $D_1:(D_2)^{>\alpha}(D_4)_{>\alpha}, D_1:(D_1)^{>\alpha}(D_4)_{>\alpha}, D_1:(D_1)_{>\alpha}(D_2)^{>\alpha}, D_1:(D_1)^{>\alpha}_{>\alpha},$\\
\indent $\bullet$ $ D_2:(D_1)_{>\alpha}(D_3)^{>\alpha}, D_2:(D_1)_{>\alpha}(D_4)^{>\alpha}, D_2:(D_3)^{>\alpha}(D_4)_{>\alpha}, D_2:(D_4)^{>\alpha}_{>\alpha},$\\
\indent $\bullet$ $ D_3:(D_2)_{>\alpha}(D_4)^{>\alpha}, D_3:(D_2)_{>\alpha}(D_3)^{>\alpha}, D_3:(D_3)_{>\alpha}(D_4)^{>\alpha}, D_3:(D_3)^{>\alpha}_{>\alpha},$\\
\indent  $\bullet$ $D_4:(D_1)^{>\alpha}(D_3)_{>\alpha}, D_4:(D_1)^{>\alpha}(D_2)_{>\alpha}, D_4:(D_2)^{>\alpha}(D_3)_{>\alpha}, D_4:(D_2)^{>\alpha}_{>\alpha}.$}
\end{definition}

\begin{definition}\emph{(}A vertex $x$ is circular\emph{)}\label{def2}
\emph{ If $x\in D_i$ such that all but at most $c\sqrt{\mu} |D_{i+1}|$ vertices in $D_{i+1}$ are out-neighbourhood of $x$ and all but at most $c\sqrt{\mu} |D_{i-1}|$ vertices in $D_{i-1}$ are in-neighbourhood of $x$, then the vertex $x$ is circular.}
\end{definition}

\begin{definition}\emph{(}A path $P$ is circular\emph{)}\label{def7}
\emph{A path $P$ is a circular path if every vertex of it is circular.}
\end{definition}

Visibly, the definition of circular is stronger than the definition of acceptable. Now we give the characterization of the extremal family $\mathcal{F}$. Here, the partition $(D_1,D_2,D_3,D_4)$ of $D$ is ordered.

\begin{definition}\emph{(}Extremal family $\mathcal{F}$\emph{)}\label{def4} %\mathcal{F}
\emph{A family $\mathcal{F}$ of oriented graphs is extremal (Fig. \ref{tu1}) if every oriented graph $D\in \mathcal{F}$ on $n$ vertices and for a sufficiently small real $\mu$, there exists a partition $(D_1,D_2,D_3,D_4)$ of $D$ satisfying the following properties:}

 \indent $\bullet$ $|D_i|=(1/4\pm 16\mu)n$, for $i\in [4]$.\\
 \indent $\bullet$ $a(D_i, D_{i+1}) > (1-800\mu)n^2/16$, for $i\in [4]$ \emph{(}modulo 4\emph{);} $a(D_i)>(1/2- 250\mu)n^2/16$, for $i\in \{1,3\}$\emph{;} and $a(D_i, D_j) > (1/2- 300\mu)n^2/16$, for $i,j\in \{2,4\}$ satisfying $i\neq j$.\\% and $e(D, B)> (1/2- 300\mu)n^2/16$.\\
 \indent $\bullet$ Every vertex is acceptable and the number of non-circular vertices is at most $100\sqrt{\mu} n$.
\end{definition}

\begin{figure}[H]
\centering    %居中       %子图居中
   \includegraphics[scale=0.5]{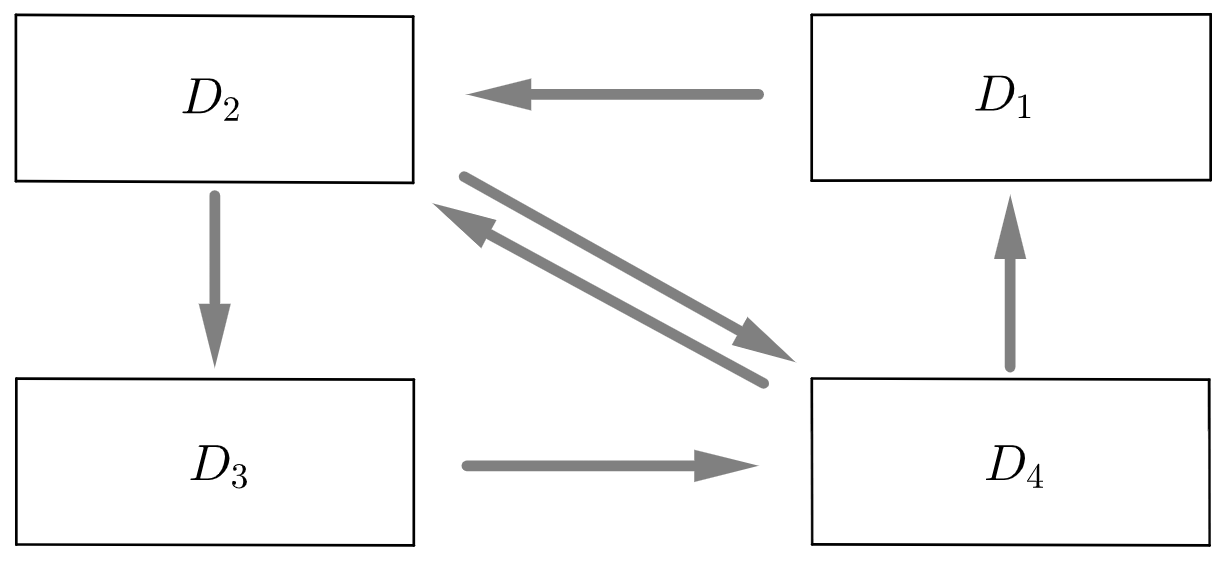}   %以pic.jpg的0.5倍大小输出
\caption{A member $D$ of extremal family $\mathcal{F}$ of order $n$, where the number of arcs from $D_i$ to $D_{i+1}$ are close to $n^2/16$, for $i\in [4]$ (shown in bold); the number of arcs from $D_2$ to $D_4$ and the number of arcs from $D_4$ to $D_2$ are close to $n^2/32$ (shown in bold).}
\label{tu1}
\end{figure}
\subsection{Diregularity lemma and Blow-up lemma}
%The digraphs considered in this paper are always simple (without loops and multiple arcs).

%\subsection{Notations and tools}
In this subsection, we collect the information we need about Diregularity Lemma and Blow-up Lemma. The density of a bipartite graph $G = (A, B)$ with vertex classes $A$ and $B$ is defined to be $d_G(A, B):= \dfrac{e_G(A, B)}{|A||B|}$. We often write $d(A, B)$ if this is unambiguous. Given $\varepsilon >0$, we say that $G$ is \emph{$\varepsilon $-regular} if for all subsets $X\subseteq A$ and $Y\subseteq B$ with $|X|>\varepsilon |A|$ and $|Y|>\varepsilon |B|$ we have that $|d(X, Y )-d(A, B)| < \varepsilon$. Given $d\in [0, 1]$ we say that $G$ is \emph{$(\varepsilon, d)$-super-regular} if it is $\varepsilon$-regular and furthermore $d_G(a)\geq (d-\varepsilon)|B|$ for all $a\in A$ and $d_G(b)\geq (d-\varepsilon)|A|$ for all $b\in B$. %(This is a slight variation of the standard definition of $(\varepsilon, d)$-super-regularity where one requires $d_G(a)\geq d|B|$ and $d_G(b)\geq d|A|$.)

In 1978, Szemer\'{e}di proposed Regular Lemma on graphs, and later Alon and Shapira \cite{Alon(2004)} extended it to the digraph version. %We will use the degree form of Diregularity Lemma which can be easily derived (see e.g. \cite{Young(2007)}) from the standard version.

\begin{lemma}\label{lem1}
\cite{Alon(2004)} \textbf{\emph{(}Diregularity Lemma\emph{)}} For every $\varepsilon \in (0, 1)$ and the number $M'$ there are numbers $M(\varepsilon,M')$ and $n_0$ such that if $D$ is a digraph on $n\geq n_0$ vertices, and $d\in [0, 1]$ is any real number, then there is a partition of the vertices of $D$ into $V_0, V_1,\ldots , V_k$ and a spanning subdigraph $D'$ of $D$ such that the following holds:\\
 \indent $\bullet$ $|V_0|\leq \varepsilon n$, $|V_1| = \cdots = |V_k|$, where $M' \leq k \leq M$,\\
 \indent $\bullet$ for each $\sigma \in \{+,-\}$, $d^{\sigma}_{D'}(x) > d^{\sigma}_D(x)- (d + \varepsilon)n$ for all vertices $x\in D$,\\
 \indent $\bullet$ for all $i = 1, \ldots , k$ the digraph $D'[V_i]$ is empty,\\
 \indent $\bullet$ for all $1\leq i, j\leq k$ with $i \neq j$ the bipartite digraph whose vertex classes are $V_i$ and $V_j$ and whose arcs are all of $A(V_i,V_j)$ arcs in $D'$ is $ \varepsilon$-regular and has density either $0$ or density at least $d$.
\end{lemma}
%In more detail, $V_1,\ldots , V_k$ are called \emph{clusters}, $V_0$ is called \emph{the exceptional set} and the vertices in $V_0$ are called \emph{exceptional vertices}.

Given $V_1,\ldots , V_k$ and a digraph $D'$, the reduced digraph $R'$ with parameters $(\varepsilon, d)$ is the digraph whose vertex set is $[k]$ and in which $ij$ is an arc if and only if the bipartite digraph whose vertex classes are $V_i$ and $V_j$ and whose arcs are all the arcs from $V_i$ to $V_j$ in $D'$ is $\varepsilon$-regular and has density at least $d$. Note that $R'$ is not necessarily an oriented graph even if $D$ is.
%It is easy to check that the reduced digraph $R'$ obtained from the regularity lemma "almost inherits" the minimum semidegree of $D$, that is $\delta^{\sigma}(R')/|R'|> \delta^{\sigma}(D)/|D|- d -2\varepsilon$, for $\sigma \in \{+,-\}$.
The next lemma shows that there is a reduced oriented graph $R \subseteq R'$ which still almost inherits the minimum semidegree and density of $D$.

\begin{lemma}\label{lem2}
\cite{Kelly(2008)} For every $\varepsilon\in (0, 1)$ there exist numbers $M' = M'(\varepsilon)$ and $n_0 = n_0(\varepsilon)$ such that the following holds. Let $d\in [0, 1]$ with $\varepsilon \leq d/2$. Let $D$ be an oriented graph of order $n\geq n_0$ and $R'$ the reduced digraph with parameters $(\varepsilon, d)$ obtained by applying Diregularity Lemma to $D$. Then $R'$ has a spanning oriented subgraph $R$ such that %\\ \indent  $(a)$
$\delta^0(R)\geq (\delta^0(D)/|D|- (3\varepsilon + d))|R|$.
%\indent $(b)$ for all disjoint sets $S, T \subseteq V(R)$  with $a_D(S^*, T^*)\geq 3dn^2$ we have $a_R(S, T) > d|R|^2$, where $S^*:=\sum_{i\in S} V_i$ and $T^*:=\sum_{i\in T} V_i $;\\
%\indent $(c)$ for every set $S \subseteq V(R)$ with $a_D(S^*)\geq 3dn^2$ we have $a_R(S) > d|R|^2$, where $S^*:=
%\sum_{i\in S} V_i$.
\end{lemma}

To find a subdigraph with the maximum degree to be bounded, the standard idea is to find a special structure in the reduced oriented graph $R$ and restore it to the subdigraph in $D$. So Blow-up Lemma of Koml\'{o}s, S\'{a}rk\"{o}zy and Szemer\'{e}di \cite{Komlos(1997)} is necessary. Notice that, in the general case, we use the Blow-up Lemma on the underling graph of $R$. The following lemma states that dense regular pairs and complete bipartite graphs behave identically for the problem of finding subgraphs with bounded degree in graphs.

%The result Pancyclic
\begin{lemma}\label{lem3}
\cite{Komlos(1997)} \textbf{\emph{(}Blow-up Lemma\emph{)}} Given a graph $F$ on $[k]$ and positive numbers $d, \Delta$, there is a positive real $\eta_0 =\eta_0(d, \Delta, k)$ such that the following holds for all positive numbers $l_1,\ldots , l_k$ and all $0 <\eta \leq \eta_0$. Let $F'$ be the graph obtained from $F$ by replacing each vertex $i\in F$ with a set $V_i$ of $l_i$ new vertices and joining all vertices in $V_i$ to all vertices in $V_j$ whenever $ij$ is an edge of $F$. Let $G'$ be a spanning subgraph of $F'$ such that for every edge $ij\in F$ the graph $(V_i, V_j)_{G'}$ is $(\eta, d)$-super-regular. Then $G'$ contains a copy of every subgraph $H$ of $F'$ with $\Delta(H)\leq \Delta$. %Moreover, this copy of $H$ in $G'$ maps the vertices of $H$ to the same sets $V_i$ as the copy of $H$ in $F'$, that is, if $h\in V(H)$ is mapped to $V_i$ by the copy of $H$ in $F'$, then it is also mapped to $V_i$ by the copy of $H$ in $G'$.
\end{lemma}

\section{Cycle-factor in oriented graphs}
Before staring the proof of Theorem \ref{main1}, we outline the main idea.

\textbf{Sketch of proof of Theorem \ref{main1}.}  Suppose that $D$ is an oriented graph as stated in Theorem \ref{main1}. Assume that Theorem \ref{main1} is false, namely, $D$ contains no some cycle-factor $\mathcal{C}$. Using the asymptotic pancyclicity (Lemma \ref{lemma1}) of $D$, we can find all disjoint short cycles in $\mathcal{C}$. Applying Diregularity lemma for the remaining digraph, there is an reduced oriented graph $R$. According to Lemmas \ref{lemma7}-\ref{lemma6}, $D$ is not an extremal oriented graph in $\mathcal{F}$ as otherwise we can find the cycle-factor $\mathcal{C}$. Further, this implies that $R$ is a robust $(\mu, 1/3)$-outexpander. Then the remaining cycles in $\mathcal{C}$ can be find by using the splitting operation (see Claim \ref{claim3}), a contradiction.

\subsection{Preliminary lemmas}
Suppose that $D$ is an oriented graph on $n$ vertices with $\delta^0(D)\geq (3n-4)/8$, where $n$ is large enough. Here, we show a sequence of lemmas that are crucial to our proof.

In 1998, Shen \cite{Shen(1998)} showed that any oriented graph $D$ on $n$ vertices with $\delta^+(D)\geq 0.355n$ contains a triangle. And we write \emph{$l$-cycle} to be a cycle with length $l$. In \cite{Kelly(2009)}, authors also showed that every oriented graph $D$ on $n$ vertices with $\delta^0(D)\geq n/3+1$ contains an $l$-cycle for every $l\in \{4,\ldots,n/10^{10}\}$. For the sake of presentation, we call a digraph is asymptotic pancyclic if it has an $l$-cycle for every $l\in \{3,4,\ldots,n/10^{10}\}$. Therefore, the following lemma is straightforward.

\begin{lemma}\label{lemma1}\textbf{\emph{(}Asymptotic Pancyclic\emph{)}}
 If $D$ is an oriented graph on $n$ vertices with $\delta^0(D)\geq 0.355n$, then $D$ is asymptotic pancyclic.
\end{lemma}%

%The following fact can easily obtained by the definition of robust $(\mu, \tau)$-outexpander.

\begin{fact}\label{fac1}
If $R$ with $\delta^0(R)\geq (3/8-3d)|R|$ is a robust $(\mu,1/3)$-outexpander, then $R$ is also a robust $(\mu, \tau)$-outexpander, where $d\ll \mu\leq \tau \leq 1/100$.
\end{fact}

\begin{proof}
%In what follows, $k$ denotes the order of $R$.
Let $|R|=k$. By the definition of robust $(\mu,1/3)$-outexpander, we have that $|N^+_R(S)|\geq |S| + \mu k$ for every $S\subseteq V(R)$ satisfying $ k/3 \leq |S| \leq 2k/3$. For a subset $S$ of $V(R)$ satisfying $\tau k\leq |S| \leq k/3$, it is easy to check that $|N^+_R(S)|\geq (3/8-3d)k\geq k/3+k/100\geq |S| + \mu k.$ And if $ 2k/3 \leq |S| \leq (1-\tau)k$, then $|S| + |N^-_R(v)| > k$, that is, $S\cap N^-_R(v)\neq \emptyset$, for any $v\in R$. And then $N^+_R(S) = V(R)$. So $$|N^+_R(S)|=|R|=(1-\tau)k+\tau k \geq |S| + \tau k\geq |S| + \mu k.$$ This implies that $R$ is a robust $(\mu, \tau)$-outexpander, as desired.
\end{proof}

%It is interest that Lemma \ref{lemma2} allows some oriented graphs without the semidegree greater than $(3n-4)/8$ are Hamiltonian., by the semidegree condition of $R$.
%\begin{lemma}\label{lemma2}
%\cite{Keevash(2009)} Let $M'$, $n_0$ be positive numbers and let $\varepsilon, d, \eta, \mu, \tau$ be positive reals such
%that $1/n_0 \ll 1/M' \ll \varepsilon \ll d \ll \mu \ll \tau \ll \eta < 1$. Let $D$ be an oriented graph on $n\geq n_0$ vertices with $\delta^0(D)\geq 2\eta n$. Let $R'$ be the reduced digraph of $D$ with parameters $(\varepsilon, d)$ and $|R'| \geq M'$. Suppose that there exists a spanning oriented subgraph $R$ of $R'$ with $\delta^0(R)\geq \eta|R|$ and such that it is a robust $(\mu, \tau)$-outexpander. Then $D$ is Hamiltonian.
%\end{lemma}
% with partition $V_0,V_1,V_2,\ldots,V_k$

\begin{lemma}\label{lemma3}
\cite{Keevash(2009)} Given $n^{-1/6}_0 \ll 1/M' \ll \varepsilon \ll d \ll \mu \ll \tau \ll \eta \leq 1$. Let $D$ be an oriented graph on $n\geq n_0$ vertices, and $L$ is a subset of $V(D)$ with cardinality at most $t\varepsilon^2 n$. Let $R'$ be the reduced digraph of $D-L$ with parameters $(\varepsilon, d)$. Suppose that $R$ is a spanning oriented subgraph of $R'$ as stated in Lemma \ref{lem2}.\\
\indent $(i)$ If $R$ is a robust $(\mu, \tau)$-outexpander and $\delta^0(D)\geq 2\eta n$, then $D$ is Hamiltonian.\\
\indent $(ii)$ If $R$ is not a robust $(\mu, 1/3)$-outexpander and $\delta^0(D)\geq (3n-4)/8$, then $D$ is a member of family $\mathcal{F}$.
\end{lemma}

 The following lemma is come from Claim 2.5-2.7 in \cite{Keevash(2009)} (see pages 20-22 for details).

\begin{lemma}\label{lemma7}
\cite{Keevash(2009)} Let $D$ be an oriented graph in $\mathcal{F}$ with a partition $(D_1,D_2,D_3,D_4)$ and order $n$. If $\delta^0(D)\geq (3n-4)/8$, then there exists an oriented graph $D^\prime$ which is obtained by contracting a few short paths or resetting a few vertices from $D$ satisfies $|D_2|=|D_4|$ in $D^\prime$.
\end{lemma}

In the proof of the next lemma, we often use an operation on contracting a path, which can be summarized as an algorithm here.
\begin{breakablealgorithm} \label{suanfa1}
		\caption{Contracting paths}
		\begin{algorithmic}[1] %每行显示行号
			\Require An oriented graph $D\in \mathcal{F}$ with a partition $(D_1,D_2,D_3,D_4)$ and disjoint paths $P_1,P_2,\ldots,P_k$, where the initial and terminal vertices of each path are located in the same part.
			\Ensure A new oriented graph $H\in \mathcal{F}$ and distinct vertices $p_1,p_2,\ldots,p_k$.
			\State Set $H_0=D$.% Set $P_i=i_1i_2\ldots i_r$ where $i_1,i_r$ belong to the same part $D_l$.
			\For {$j=1$}	
			\While{$j\neq k+1$}	
			\State Set $P_j=j_1j_2\ldots j_r$ where $j_1,j_r\in D_l$, for some $l\in [4]$.
			\State Construct the new oriented graph $H_j$ by contracting the path $P_j$ from $H_{j-1}$ into a new vertex $p_j$ with $N^{-}_{H_j}(p_j)=N^{-}_{H_{j-1}}(i_1)\cap D_{l-1}$ and $N^{+}_{H_j}(p_j)=N^{+}_{H_{j-1}}(i_r)\cap D_{l+1}$, and then put $p_j$ into the part $D_l$.
			\State Let $j=j+1$.
			\EndWhile
		    \EndFor
		    \State Set $H=H_k$.
		\end{algorithmic}
	\end{breakablealgorithm}

Lemma \ref{lemma6} states that there are some oriented graphs in extremal family $\mathcal{F}$ contains any cycle-factor as desired.
%In addition, we will state every oriented graph in $\mathcal{F}$ either contains a cycle-factor as desired or we can arrange that every vertex is good so that we will claim the semidegree of the oriented graph less than $(3n-4)/8$.  and such that $|R'| \geq M'$

\begin{lemma}\label{lemma6}
Given a constant $t$ and a positive real $\mu$ such that $\mu \ll 1/(10^{10}\cdot t^2)$, there exists a number $n_0$ so that every oriented graph $D\in \mathcal{F}$ \emph{(Fig. \ref{tu1})} on $n\geq n_0$ vertices with a partition $(D_1,D_2,D_3,D_4)$ and $|D_2| = |D_4|$ contains $t$ disjoint cycles of lengths $n_1,\ldots , n_t$, where $n = n_1 +\cdots + n_t$ is any positive integer partition of $n$.
\end{lemma}

\begin{proof}
On the contrary, suppose that $D$ has no cycle-factor $\mathcal{C}=\{C_1,C_2,\ldots, C_t\}$ with $|V(C_i)|=n_i$ for each $i\in [t]$ and $n=\sum^t_{i=1}n_i$. For convenience, we represent a path using a sequence of the parts corresponding to each vertex along the path. We call a path (resp. cycle) is circular, if the successive vertices of the path (resp. cycle) lie in successive classes. And let the index $i$ be always taken modulo 4, in this lemma. Without loss of generality, suppose that $n_1\geq n_i$ for $i\in [t]$. Also, assume that each of $n_2,n_3,\ldots, n_{l^\prime}$ is equal to 3 and %each $n_i$ in $\{n_{l^\prime+1}, n_{l^\prime+2}, \ldots, n_{l^\prime+l_0}\}$ satisfys $n_i\equiv 0$ $(mod$ 4$)$, each $n_i$ in $\{n_{l^\prime+l_0+1}, n_{l^\prime+l_0+2}, \ldots, n_{l^\prime+l_0+l_1}\}$ satisfys $n_i\equiv 1$ $(mod$ 4$)$, each $n_i$ in $\{n_{l^\prime+l_0+l_1+1}, \ldots,$ $n_{l^\prime+l_0+l_1+l_2}\}$ satisfys $n_i\equiv 2$ $(mod$ 4$)$, and each $n_i$ in $\{n_{l^\prime+l_0+l_1+l_2+1}, \ldots, n_{l^\prime+l_0+l_1+l_2+l_3}\}$ satisfys $n_i\equiv 3$ $(mod$ 4$)$.

$$n_i=
\begin{cases}
=0 \quad i\in \{l^\prime+1, l^\prime+2, \ldots, l_0\};\\
=1 \quad i\in \{l_0+1, l_0+2, \ldots, l_1\};\\
=2 \quad i\in \{l_1+1, l_1+2, \ldots, l_2\};\\
=3 \quad i\in \{l_2+1, l_2+2, \ldots, l_3\}.\\
\end{cases}
(\text{mod\ }4)$$

%To construct a contradiction, we will find the cycle factor $\mathcal{C}$ in $D$. We first seek out and remove all triangles in the cycle factor $\mathcal{C}$. And then to find cycles of lengths $l^\prime+1, l^\prime+2, \ldots, l_3$, contract $l_3-l_0$ appropriate disjoint paths, resulting in oriented graph $H_1$. Then remove the set of contraction vertices $Z$ from $H_1$ to obtain oriented graph $H_2$. To ensure that all vertices are circular, contract additional disjoint paths to obtained a new oriented graph, denoted $H_3$. We added the set $Z$ back into $H_3$ to obtain oriented graph $H_4$. There are disjoint cycles in $H_4$ and these cycles were the desired cycles of lengths $l^\prime+1, l^\prime+2, \ldots, l_3$ in $D$. Finally, for the remaining oriented graph $H_5$, we will demonstrate that $H_5$ is Hamiltonian, that is, there is the cycle factor $\mathcal{C}$ in $D$, a contradiction.

Firstly, we will find all the triangles which are needed. Constituting $l^\prime-1$ disjoint triangles such that every triangle has a vertex in $D_2$, a vertex in $D_3$ and a vertex in $D_4$. Recall that there is a few vertices are non-circular and $a(D_2,D_4)\geq (1/2- 300\mu)$, which implies that there are $l^\prime-1$ disjoint arcs from $D_2$ to $D_4$ with endvertices are all circularity. Owing to the definition of circular, we can obtain $l^\prime-1$ disjoint triangles as desired. From now on, define $S$ to be the set of $l^\prime-1$ triangles. Let $D^\prime=D-S$, we conclude that $D^\prime \in \mathcal{F}$ and $|D_2|=|D_4|$ in $D^\prime$, since $l^\prime$ is a constant.

In order to find cycles with given lengths, we will find $l_3-l_0$ suitable disjoint paths.
% so that, for each contraction vertex, we just take a cycle includes the vertex whose length is $4k$, for some $k\in \mathbb{N}$.
Select $l_1-l_0$ disjoint arcs whose two endvertices are circular vertices in $D_1$ as $l_1-l_0$ paths. And pick $l_2-l_1$ disjoint paths such that every path is a circular path $D_3D_3D_3$. And then find $l_3-l_2$ disjoint circular paths, where each path is shaped like $D_2D_4D_1D_2$. Note that we can ensure that all of above paths are disjoint, since the numbers of arcs in $D_1,D_3$, from $D_i$ to $D_{i+1}$ and from $D_2$ to $D_4$ are large enough. These yield that a path system $\mathcal{P}=\{P_{l_0+1},P_{l_0+2},\ldots, P_{l_3}\}$. Now, the new oriented graph $H_1$ and contraction vertices $p_{l_0+1}, p_{l_0+2}, \ldots, p_{l_3}$ are obtained by applying Algorithm \ref{suanfa1} with the oriented graph $D^\prime$ and the path system $\mathcal{P}$. It is easy to check that all contraction vertices are circular and $|D_2|=|D_4|$ in $H_1$. Set $Z=\{p_{l_0+1}, p_{l_0+2},$ $\ldots, p_{l_3}\}$. Removing the vertices in $Z$ and arcs which adjacent with $Z$ from $H_1$, we can obtain a new oriented graph $H_2$.

Actually, we need all vertices to be circular. Thereby, our purpose is to find some disjoint short paths containing all non-circular vertices so that all vertices are circular after contracting these paths. Assume $v_1,\ldots, v_r$ are non-circular vertices in $H_2$. Next, we will explain that the following disjoint paths $Q_1,Q_2,\ldots,Q_r$ are paths that we want. For each $v_i$, choose a circular out-neighbour $v^+_i$ and a circular in-neighbour $v^-_i$ so that they are distinct. Then there exist disjoint paths $Q_1,Q_2,\ldots,Q_r$ such that for each $i\in [r]$, the following hold.\\
 \indent $ \bullet$ The path $Q_i$ starting at $v^-_iv_iv^+_i$ and ending at a circular vertex which lies in the same class as $v^-_i$ with length at most 7.\\
 \indent $ \bullet$ The path $Q_i$ is a circular path.

Applying Algorithm \ref{suanfa1} with the oriented graph $H_2$ and the paths $Q_1,Q_2,\ldots,Q_r$, we obtain a new oriented graph $H_3$ and contraction vertices $q_1,q_2,\ldots,q_r$. For each $j\in [r]$, since the endvertices of $Q_i$ are circular vertices, vertex $q_j$ is circular, which is easy to check by the construction of $q_j$. Now every vertex of $H_3$ is circular.

We construct a new oriented graph $H_4$ by adding the vertex set $Z$ to $H_3$ and arcs between $Z$ and $V(H_3)$ in $A(H_1)$. Hence $|D_2| = |D_4|$ in $H_4$ and we still have that
\begin{equation}\label{1}
\begin{aligned}
||D_i|-|D_j||\leq |(1/4+ 16\mu)n-(1/4-16\mu)n|+t+r\leq 300\sqrt{\mu}n.
\end{aligned}
\end{equation}
Recall that $a(D_2, D_4), a(D_4, D_2) > (1/2- 300\mu)n^2/16$ in $D$, which gives a matching $M_0$ of size $2\times 10^4\sqrt{\mu} n$, where $10^4\sqrt{\mu} n$ arcs are from $D_2$ to $D_4$ and $10^4\sqrt{\mu} n$ arcs are from $D_4$ to $D_2$ in $H_4$.

Next, we begin by finding disjoint cycles $C_{l^\prime+1},C_{l^\prime+2},\ldots, C_{t}$. Due to the definition of circularity and $\mu \ll 1/(10^{10}\cdot t^2)$, for any two circular vertices $y\in D_{j-2}, x\in D_j$, we actually get that
\begin{equation}\label{2}
\begin{aligned}
|N^+(y)\cap N^-(x)\cap D_{j-1}|\geq n/4-6000\sqrt{\mu} n\geq n/4-n/4t+100\sqrt{\mu} n+2\times 10^4\sqrt{\mu} n.
\end{aligned}
\end{equation}
This implies that there are at least $n/4-n/4t+2\times 10^4\sqrt{\mu} n$ circular vertices in $N^+(y)\cap N^-(x)\cap D_{j-1}$. However, since $n_1\geq n_i$ for $1\leq i \leq t$, we get $n_1\geq n/t$. Thus $n-n_1\leq n-n/t$. Also, it is clear from (\ref{2}) that there are $l_1+l_2+l_3$ disjoint cycles $C^\prime_{l^\prime+1},\ldots,C^\prime_t$ satisfying the following statements.\\
\indent $ \bullet$ For each $i\in [l^\prime+1,t]$, $C^\prime_i$ with length $\lfloor n_i/4\rfloor \times 4$ is a circular cycle which avoids the vertices in the matching $M_0$.\\
\indent $ \bullet$ For each $i\in [l_0+1,t]$, $C^\prime_i$ contains the corresponding contraction vertex $p_i$.\\
Consider each $C^\prime_i$, if we replace the contraction vertex $p_i$ with the path $P_i$, we can obtain a cycle of length $n_i$ in $D$, for $i\in [l_0+1,t]$. Set $U=V(C^\prime_{l^\prime+1}\cup C^\prime_{l^\prime+2}\cup \cdots \cup C^\prime_t)$ and $H_5:=H_4-U$. Clearly, $M_0\subseteq H_5$. Moreover, $|D_2|=|D_4|$ still holds in $H_5$.
%Recall that  has no cycle-factor $\mathcal{C}$, then $H_5$ is not Hamiltonian.Since the construction of $\mathcal{F}$ implies that almost all pairs of circular vertices in $D_3$ are joined by an edge.Note that the $D_3D_3D_4D_1D_2$ fragment uses one circular vertex in $D_1,D_2,D_4$, two vertices in $D_3$.

Recall that $D$ does not have a cycle-factor $\mathcal{C}$, then $H_5$ cannot be Hamiltonian. However, we will claim that $H_5$ is Hamiltonian, which contradicts the hypothesis and proves the lemma. In order to apply Blow-up Lemma, we also need $|D_1|=|D_2|=|D_3|=|D_4|$. we can do this by finding two suitable disjoint paths to contract. Without loss of generality, assume that $|D_1| < |D_3|$ in $H_5$. Let $s := |D_3|- |D_1|$ in $H_5$, (\ref{1}) implies that $s<300\sqrt{\mu} n$. Due to the structure of $D$, almost all pairs of vertices in $D_3$ are connected by an arc. Let $H_6=H_5-V(M_0)$. Thus we greedily find a path $R_{1,3}$ of the form
 $$D_3D_3D_4D_1D_2 \cdots D_3D_3D_4D_1D_2D_3,$$
 where the fragment $D_3D_3D_4D_1D_2$ consists of circular vertices and it repeats $s$ times. So $R_{1,3}$ starts with an arc between two circular vertices in $D_3$. Let $H_7$ be the oriented graph by applying Algorithm \ref{suanfa1} with the oriented graph $H_6$ and the path $R_{1,3}$. Then add the vertex set $V(M_0)$ and arcs between $V(M_0)$ and $V(H_7)$ in $A(H_5)$, assuming there is no confusion, we still call it $H_7$. It follows that $|D_1| = |D_3|$, $|D_2| = |D_4|$ in $H_7$ and all vertices of $H_7$ are still circular. Without loss of generality, suppose that we have $|D_2| > |D_1|$ in $H_7$. Let $s := |D_2|- |D_1|$,  (\ref{1}) implies that $s<300\sqrt{\mu} n$. By using the arcs in the matching $M_0$, we can find a path $R_{2,1}$ of the form
$$D_2D_4D_1D_2D_3D_4D_2D_3D_4D_1 \cdots D_2D_4D_1D_2D_3D_4D_2D_3D_4D_1D_2,$$ where the fragment $D_2D_4D_1D_2D_3D_4D_2D_3D_4D_1$ appears repeatedly for $s$ times. It is easy to check that $D_7$ has such a path. Likewise, by contracting $R_{2,1}$, we obtain an oriented graph $H_8$. In $H_8$, we get that
\begin{equation*}
\begin{aligned}
|D_1| = |D_2| = |D_3| = |D_4|&\geq (n_1-6r-5\times 300\sqrt{\mu}n-10\times 300\sqrt{\mu}n)/4\\
&\geq (n_1-10^4\sqrt{\mu}n)/4.
\end{aligned}
\end{equation*}
Note that, all vertices are still circular in $H_8$. Namely, for every vertex $v\in D_i$, there are at least $(1-ct\sqrt{\mu}) |D_{i+1}|$ out-neighbours in $D_{i+1}$.
%$$|D_1| = |D_2| = |D_3| = |D_4|\geq (n_1-10^4\sqrt{\mu}n)/4.$$

Let $F'$ be the 4-partite graph with vertex classes $D_1,D_2,D_3, D_4$ in $H_8$, where the bipartite graphs induced by $(D_i,D_{i+1})$ are all complete. Clearly $F'$ is Hamiltonian. On the other hand, assume that $G$ is the underlying graph corresponding to the set of edges oriented from $D_i$ to $D_{i+1}$ in $H_8$, for $i\in [4]$.
Pick $\eta$ with $c t \sqrt{\mu}\ll \eta^2$. Since all vertices of $H_8$ are still circular, for any subset $X\subseteq D_i$ and $Y\subseteq D_{i+1}$ with $|X|\geq \eta |D_i|$, $|Y|\geq \eta |D_{i+1}|$, it follows that $$d(X,Y)\geq \dfrac{(|Y|-c t\sqrt{\mu}|D_{i+1}|)|X|}{|X||Y|}\geq 1-\dfrac{c t\sqrt{\mu}}{\eta}.$$ Therefore, each pair $(D_i,D_{i+1})$ is $\eta$-regular pair in $G$. Further, each pair $(D_i,D_{i+1})$ is $(\eta,1)$-super-regular pair as all vertices of $H_8$ are circular. Also, $G$ is simple. So we can apply Lemma \ref{lem3} with $k = 4$, $\Delta = 2$ to get a Hamiltonian cycle in $G$. Since the construction of $G$, $H_8$ is Hamiltonian. Recall that $H_8$ is obtained by $H_5$ contracting two paths, this implies that $H_5$ is Hamiltonian. This contradiction completes the proof of the lemma.
\end{proof}

%For more details of the following lemma, we refer the reader to [7].

\subsection{Proof of Theorem \ref{main1} }
Define constants $M', \varepsilon, d, \mu, \tau, \eta, n_0$ satisfying
\begin{equation}\label{5}
\begin{aligned}
n^{-1/6}_0\ll 1/M'\ll \varepsilon^2 \ll d \ll \mu \ll \tau \ll \eta \ll 1/t.
\end{aligned}
\end{equation}
 Suppose that $D$ is an oriented graph on $n\geq n_0$ vertices with $\delta^0(D)\geq (3n-4)/8$. Assume that Theorem \ref{main1} is false. Namely, $D$ contains no cycle-factor $\mathcal{C}$ with cycles $C_1,C_2,\ldots, C_t$ whose orders are $n_1,n_2,\ldots,n_t$, respectively. Without loss of generality, assume that $n_1\geq n_2\geq \cdots n_j\geq \varepsilon^2 n > n_{j+1}\geq \cdots\geq n_t$. Let's start by finding disjoint cycles whose lengths less than $\varepsilon^2 n$. This can be done by applying Lemma \ref{lemma1}. These yield $t-j$ disjoint cycles $C_{j+1}, C_{j+2}, \ldots, C_t$ where each cycle $C_i$ has the length $n_i<\varepsilon^2 n$. Define $D^*$ to be the oriented graph $D-\bigcup^t_{i=j+1} C_i$. It is easy to check that the semidegree of $D^*$ at least $(3/8-t \varepsilon^2 )n$. Apply Diregularity Lemma (Lemma \ref{lem1}) with parameters
 $\varepsilon^2, d, M^\prime$ for the digraph $D^*$ to obtain a partition $V_0,V_1,V_2,\ldots,V_k$ and a reduced oriented graph $R$ by Lemma \ref{lem2}. Further, we have $\delta^0(R)\geq (3/8-t\varepsilon^2-d-3\varepsilon)|R|\geq (3/8-3d)|R|$, as $(\ref{5})$.%Lemma \ref{lem4} $(i)$ states that $D^*$ is a robust $(\mu-t\cdot \varepsilon, \tau)$-outexpander.

\begin{claim}\label{claim2}
$R$ is a robust $(\mu, \tau)$-outexpander.
\end{claim}

\begin{proof}
Suppose that $R$ is not a robust $(\mu, 1/3)$-outexpander. It follows from Lemma \ref{lemma3} that $D$ is an extremal oriented graph. Owing to the semidegree of $D$ and Lemma \ref{lemma7}, there exists an oriented graph obtained by contracting a few short paths or resetting a few vertices from $D$ satisfies the hypothesis of Lemma \ref{lemma6}. Namely, there is a cycle-factor with lengths $n_1,n_2,\ldots,n_t$, a contradiction. Thereby, $R$ is a robust $(\mu, 1/3)$-outexpander. Combining with Fact \ref{fac1}, $R$ is a robust $(\mu, \tau)$-outexpander.
\end{proof}

In the following, we will apply Lemma \ref{lemma3} to find disjoint cycles with lengths $n_1,n_2,\ldots,n_j$. To do this, we split $V(D^*)$ into $S_1, S_2,\ldots,S_j$ with almost inherit the semidegree condition of $D^*$. In this process, "Chernoff bound" is essential, which states that $\mathbb{P}(|X -\mathbb{E}X| > a) < e^{-a^2/(3\cdot \mathbb{E}X)}$, where $X$ is the hypergeometric random variable and $\mathbb{E}X$ is the expectation of $X$. Set $\xi_i=\lfloor\dfrac{n_i}{k}\rfloor\cdot \dfrac{1}{|V_i|}$. Namely, $\xi_i>\varepsilon^2$, for each $i\in [j]$.
\begin{claim}\label{claim3}
There exists a partition of $V(D^*)$ into $j$ sets $(S_1, S_2,\ldots,S_j)$ such that the following properties hold.\\
\indent \emph{(}$i$\emph{)} $|S_i|=n_i\geq \varepsilon^2 n$ and $\delta^{0}(D[S_i])\geq 2\eta |S_i|$, for $i\in [j]$.\\
\indent \emph{(}$ii$\emph{)} For each $i\in [j]$, there is an oriented reduced graph $R_i$ with parameters $(\varepsilon^2, 5d/6)$ corresponding to the partition $S_{i,1},\cdots ,S_{i,k}$ of $S_i$. Moreover, every $R_i$ is isomorphic with $R$.
\end{claim}

% as $n \rightarrow \infty$ \rightarrow 0
%we will show that there is a partition $(S_{1,l},S_{2,l},\ldots,S_{j,l})$ of $V_l$, and there is a partition $(V_{1,0},V_{2,0},\ldots, V_{j,0})$ of $V_0$, so that $\delta^{0}(D[S_i])\geq (3/8-\gamma)|S_i|$ where $S_i=\bigcup_{l\in [k]}S_{i,l}\cup V_{i,0}$. C
\begin{proof}
($i$). For each $i\in [j]$, we will show that $S_i=\bigcup_{l\in [k]}S_{i,l}\cup V_{i,0}$, where $S_{i,l}$ is a subset of $V_i$ with order $\xi_i |V_i|$ and $V_{i,0}$ is a subset of $V_0$ with order $n_i-\xi_i k|V_i|$, so $|S_i|=n_i$. For each $V_l$, consider a random partition of $V_l$ into $j$ sets $S_{1,l},S_{2,l},\ldots,S_{j,l}$. By assigning every vertex $x\in V_l$ to $S_{i,l}$ with probability $\xi_i$ independently. For every vertex $v\in V(D^*)$, let the random variable $A^+_{v,i}$ (resp., $A^-_{v,i}$) calculate the number of out-neighbours (resp., in-neighbours) of $v$ in $S_{i,l}$. It is not hard to see that
$$ \mathbb{E}A^+_{v,i} = \dfrac{d^+_{V_l}(v)}{|V_l|}\cdot |S_{i,l}|= \xi_i d^+_{V_l}(v).$$
Owing to Chernoff bound, this yields that
\begin{equation*}
\begin{aligned}
\mathbb{P}(A^+_{v,i}-\mathbb{E}A^+_{v,i}< -n^{2/3})< e^{-\dfrac{n^{4/3}}{3\cdot d^+_{V_l}(v)}}.
\end{aligned}
\end{equation*}
This implies that $$\sum_{v\in V(D^*)}\mathbb{P}(A^+_{v,i}-\mathbb{E}A^+_{v,i}< -n^{2/3}) <n e^{-\dfrac{n^{4/3}}{3\cdot d^+_{V_l}(v)}}.$$ Analogously we can get that $$\sum_{v\in V(D^*)}\mathbb{P}(A^-_{v,i}-\mathbb{E}A^-_{v,i}< -n^{2/3}) <n e^{-\dfrac{n^{4/3}}{3\cdot d^+_{V_l}(v)}}.$$
%Namely, standard Chernoff bound implies that the probability that a vertex $v\in V(D^*)$ whose semidegree in $S_{i,l}$ is too large or too small is exponentially small in $n$.
So with non-zero probability we obtain that $d^{\sigma}_{S_{i,l}}(v) \geq \xi_i d^{\sigma}_{V_l}(v)-n^{2/3}$ for every vertex $v\in V(D^*)$, every $i\in [k]$ and $\sigma \in \{+,-\}$. For each $i\in [j]$, arbitrarily pick $n_i-\xi_i k|V_l|$ vertices from the set $V_0$, and form the set $S_i$ with the set $\bigcup^k_{l=1}S_{i,l}$. For $\sigma\in \{-,+\}$, (\ref{5}) and $|S_i|=n_i\geq \varepsilon^2 n$ yield that
$$\delta^{\sigma}(D^*[S_i])\geq (3/8-t\cdot \varepsilon^2)|S_i|-n^{2/3}
\geq (3/8-t\cdot \varepsilon^2-n^{-1/3}/\varepsilon^2)|S_i|>2\eta |S_i|,  \text{ as } n \rightarrow \infty.$$

($ii$). For every edge $l_1l_2\in E(R)$, it follows from the construction of $R$ that $A[V_{l_1}, V_{l_2}]$ is $\varepsilon^2$-regular pair with density at least $d$. From the definition of regularity and $|S_{i,l_1}|\geq \varepsilon |V_{l_1}|,$ $|S_{i,l_2}|\geq  \varepsilon |V_{l_2}|$, it can be concluded that $|d(S_{i,l_1}, S_{i,l_2})-d(V_{l_1}, V_{l_2})| < \varepsilon^2$, that is $d(S_{i,l_1}, S_{i,l_2})\geq d- \varepsilon^2\geq 5d/6$. On the other hand, there is no arc from $S_{i,l_1}$ to $S_{i,l_2}$ whenever $l_1l_2\notin E(R)$. Hence, there is an oriented reduced graph $R_i$ with parameters $(\varepsilon^2, 5d/6)$ corresponding to the partition $S_{i,1},\cdots ,S_{i,k}$ of $S_i$. Clearly, $V(R_i)=V(R)$. These analyses make it obvious that $R_i$ is isomorphic with $R$.
\end{proof}

Together with Lemma \ref{lemma3} ($i$), each oriented graph $D^*(S_i)$ is Hamiltonian. Thereby, we find disjoint cycles of lengths $n_1,n_2,\ldots,n_{t}$, a contradiction. This completes the proof.
\hfill $\Box$
%%In more detail, we need to further classify acceptable vertices.

\section{Strongly Hamiltonian-connected in oriented graphs}
Before proving Theorem \ref{cor2}, we show a key lemma to the proof of Theorem \ref{cor2}. Define $\beta=(1/100 + c \sqrt{\mu})n/4$.
%\begin{definition}\emph{(}A vertex is good\emph{)}\label{def4}
When $|D_2|>|D_4|$, we call a vertex \emph{good} if it is acceptable, and it belongs to $D_4$ or has one of the properties
$D_1:(D_2)_{<\beta}(D_3)_{<\beta}$, $D_2:(D_1)^{<\beta}(D_2)^{<\beta}_{<\beta}(D_3)_{<\beta}$, $D_3 :(D_1)^{<\beta}(D_2)^{<\beta}$. Similarly, when $|D_2|<|D_4|$, we call a vertex \emph{good} if it is acceptable, and the vertex in $D_2$ or has one of the properties $D_1:(D_2)_{<\beta}(D_3)_{<\beta}$, $D_3 :(D_1)^{<\beta}(D_2)^{<\beta}$ or $D_4:(D_1)_{<\beta}(D_4)^{<\beta}_{<\beta}(D_3)^{<\beta}$. Naturally, if a vertex is not good, we call it a \emph{bad} vertex. The idea of proving Lemma \ref{main2} is very similar to proving Theorem \ref{main1}, so here we focus on the placement of the vertex $x$. Here, we still denote $\alpha=(1/100 - c\sqrt{\mu})n/4$.

\begin{lemma}\label{main2}
Suppose $D$ is an oriented graph on $n\geq n_0$ vertices with $\delta^0(D)\geq 3n/8$, where $n_0$ is some integer. The oriented graph $D^\prime$ obtained from $D$ by adding a vertex $x$ such that the semidegree of $x$ in $D^\prime$ is at least $4\alpha$, then $D^\prime$ is Hamiltonian.
\end{lemma}

\begin{proof}	
Suppose, contrary to our lemma, that $D^\prime$ is not Hamiltonian. Define constants $M', \varepsilon, d, \mu, \tau, \eta, n_0$ satisfying
$n^{-1/6}_0\ll 1/M'\ll \varepsilon \ll d \ll \mu \ll \tau \ll 1/100.$
%\begin{equation}\label{6}
%\begin{aligned}
%
%\end{aligned}
%\end{equation}
Applying Diregularity Lemma with parameters $\varepsilon^2, d, M^\prime$ for the digraph $D$, this yields a partition $V_0,V_1,\ldots,V_k$ and there is a corresponding reduced oriented graph $R$ with $\delta^0(R)\geq (3/8-3d)|R|$ by Lemma \ref{lem2}. If $R$ is a robust $(\mu, 1/3)$-outexpander, then $R$ is a robust $(\mu, \tau)$-outexpander by Fact \ref{fac1}. Put $x$ into the exception set $V_0$. Lemma \ref{main2} holds as Lemma \ref{lemma3} ($i$), a contradiction. Otherwise, $R$ is not a robust $(\mu, 1/3)$-outexpander, then $D\in \mathcal{F}$ with a partition $(D_1,D_2,D_3,D_4)$ by Lemma \ref{lemma3} ($ii$).

%\begin{claim}\label{claim1}
%There is no part $D_l$ such that $|D_2|=|D_4|$ in $D^\prime$ after placing the vertex $x$ into $D_l$ as an acceptable vertex, $l\in [4]$.
%\end{claim}
% and only resetting at most $300\sqrt{\mu}n$ vertices

 %We now break the proof up into two cases, depending on whether or not $|D_2|=|D_4|$ in $D$.

%\textbf{Case 1. When $|D_2|\neq |D_4|$ in $D$.}

%Without lose of generality, suppose that $|D_2|> |D_4|$ in $D$. There exists a $D_k$ such that the vertex $x$ is a good vertex in $D_k$.
\begin{claim}\label{claim1}
There is a part $D_r$ such that we can put $x$ into $D_r$ to be an acceptable vertex, $r\in [4]$. %Moreover, $|D_2|\neq |D_4|$ in $D^\prime$ after putting the vertex $x$ into $D_r$.
\end{claim}

\begin{proof}
Owing to the semidegree of $x$, there exist two parts $D_i$ and $D_j$ such that $|N^+(x)\cap D_i|>\alpha$ and $|N^-(x)\cap D_j|>\alpha$, $i,j\in [4]$. Therefore, the choice of $D_r$ depends on the following circumstances:

$\bullet$ When $i\in \{1,2\}$, if $j\in \{1,4\}$ then $r=1$, otherwise, $r=4$.

$\bullet$ When $i\in \{3,4\}$, if $j\in \{2,3\}$ then $r=3$, otherwise, $r=2$.\\
It is easy to check that $x$ is an acceptable vertex in $D_r$.
\end{proof}
%By Claim \ref{claim1}, $|D_2|\neq |D_4|$ in $D^\prime$ after putting the vertex $x$ into $D_r$.

Now, all vertices in $D^\prime$ are acceptable. In addition, it is not difficult to find $D^\prime$ belongs to $\mathcal{F}$. Indeed, if $|D_2|=|D_4|$ in $D^\prime$ by few arrangement, then $D^\prime$ satisfies the hypothesis of Lemma \ref{lemma6}. Thus, $D^\prime$ is Hamiltonian, a contradiction. Therefore, $|D_2|\neq |D_4|$ always hold in the following procedures as we can't move more than $2(|D_2|-|D_4|)$ vertices. Without lose of generality, suppose that $|D_2|> |D_4|$ in $D^\prime$. In the following, we only need to prove that there is a contradiction with the semidegree of $D$. Its proof is the same as Claims 2.5-2.7 in \cite{Keevash(2009)}, but for the sake of completeness of proof, we still give the complete proof here.

\begin{claim}\label{claim5}
\cite{Keevash(2009)} All vertices in $D^\prime$ will become good by some arrangement. And then there is a contradiction with the semidegree of $D$.
\end{claim}

\begin{proof}
Set $s=|D_2|-|D_4|$. Note that we put a vertex $v$ into $D_4$ to be an acceptable vertex, then $v$ is a good vertex by the definition of good vertex.
% It only needs to consider the cases which we put $x$ into a part $D_l(\neq D_4)$.

\begin{subclaim}\label{claim33}
There are at most $s-1$ bad vertices in $D_1 \cup D_3$. Moreover, we can arrange them to be good vertices by moving them to $D_4$.
\end{subclaim}

\begin{proof}
If $v$ is a bad vertex in $D_1$, we can obtain that $|N^-(v)\cap D_2|\geq \beta$ or $|N^-(v)\cap D_3|\geq \beta$. It follows from $v$ is an acceptable vertex that $|N^+(v)\cap D_1|\geq \alpha$ or $|N^+(v)\cap D_2|\geq \alpha$. It is easy to check that we move the vertex $v$ to $D_4$ then $v$ is an acceptable vertex. Furthermore, it is good. Similarly for the case $v\in D_3$. Thus, if there are $s$ bad vertices in $D_1\cup D_3$, then we can arrange them to be good vertices by moving them to $D_4$. This yields that $|D_2|=|D_4|$, a contradiction.
\end{proof}

\begin{subclaim}\label{claim4}
 All bad vertices in $D_2$ will become acceptable in $D_1\cup D_3$ by some arrangement.
\end{subclaim}

\begin{proof}
Suppose that $v$ is not a good vertex in $D_2$. Then $v$ satisfies at least one of properties $D_2:(D_1)^{>\beta}$, $D_2:(D_2)^{>\beta}$, $D_2:(D_2)_{>\beta}$ and $D_2:(D_3)_{>\beta}$. By the acceptability of $v$, we get that $v$ has a large in-neighbourhood in $D_1$ or $D_4$ and a large out-neighbourhood in $D_2$ or $D_3$. Hence, if $v$ has properties $D_2:(D_1)^{>\beta}$ or $D_2:(D_2)^{>\beta}$, then it is also an acceptable vertex in $D_1$. So we move the vertex $v$ to the part $D_1$. And if $v$ has properties $D_2:(D_2)_{>\beta}$ or $D_2:(D_3)_{>\beta}$, then $v$ is also an acceptable vertex in $D_3$. In this case, we move the vertex $v$ to $D_3$. By Subclaim \ref{claim33}, we can arrange $v$ to be good in some part $D_k$.
\end{proof}

%So far, the vertex $x$ is a good vertex by our arrangement. Also, $|D_2|\neq |D_4|$ in $D^\prime$. Without loss of generality, assume $|D_2|>|D_4|$ . Indeed, the above analysis also shows that  All bad vertices in $D_2$ will become good by some arrangement as in Claim \ref{claim2}.
Note that $|D_2|\neq |D_4|$ during the whole process. However, it may happen that $|D_2|- |D_4|$ goes from $+1$ to $-1$ if a vertex $v$ is moved from $D_2$ to $D_4$. In this case we can put $v$ into $D_1$ or $D_3$ to be an acceptable vertex, by Claim \ref{claim4}. This yields that $|D_2|=|D_4|$ in $D^\prime$, a contradiction. Thereby all vertices are good by our arrangement and $|D_2|>|D_4|$ in $D^\prime$.

Next, we will show that there is no arc from $D_3$ to $D_2$ and from $D_2\cup D_3$ to $D_1$ and in $D_2$. Construct a spanning oriented graph $H^{\prime}$ of $D^{\prime}$ whose arc set is $A(H^{\prime})=A(D_2, D_1)\cup A(D_2)\cup A(D_3, D_1)\cup A(D_3, D_2)$. Let $M$ be a maximum matching in $H^{\prime}$.
For each $i\in [3]$, suppose $L_i=V(M)\cap D_i$.
If $a(M)\geq s$, we could extend $s$ arcs of $M$ to $s$ disjoint paths such that the difference between $|D_2|$ and $|D_4|$ decreased by $s$ by contracting these paths. Then the oriented graph $D^{\prime \prime}$ obtained by contracting these paths from $D^{\prime}$ has $|D_2|=|D_4|$. By Lemma \ref{lemma6}, $D^{\prime \prime}$ is Hamiltonian. Further, $D^{\prime}$ is Hamiltonian, a contradiction. Hence, $a(M)< s$.

Then we will claim that $a(M) = 0$. Assume to the contrary that $a(M)\geq 1$. The maximality of $M$ and the definition of the good vertex imply
\begin{equation*}\label{9}
\begin{aligned}
a(D_3, D_1)&\leq a(L_3, D_1) + a(D_3, L_1 )\\
&\leq (|L_3| + |L_1|)\beta \leq a(M)|D_1|/45.
\end{aligned}
\end{equation*}
Similarly, we also can obtain that
$$a(D_2,D_1)\leq a(L_2, D_1) + a(D_2, L_1) < a(M)|D_1|/45,$$
$$a(D_3,D_2)\leq a(L_3, D_2) + a(D_3, L_2 ) < a(M)|D_3|/45,$$
$$a(D_2)\leq |L_2|\cdot 2\beta <2a(M)|D_2|/45.$$

Hence, it follows that
$$\sum_{v\in D_1} d^-(v)\leq \dfrac{|D_1|\cdot (|D_1-1|)}{2}+ 2a(M)|D_1|/45+ |D_1|\cdot |D_4|,$$
$$\sum_{v\in D_2} d(v)\leq (|D_1|+|D_3|)\cdot |D_2| + 4a(M)|D_2|/45+|D_2|\cdot |D_4|,$$
$$\sum_{v\in D_3} d^+(v)\leq \dfrac{|D_3|\cdot (|D_3-1|)}{2}+ 2a(M)|D_3|/45+ |D_3|\cdot |D_4|.$$

Without loss of generality, assume that the vertex $x$ is a good vertex in $D_1$. By pigeonhole principle, there are three vertices $u\in D_1, v\in D_2, w\in D_3$ with $u\neq x$ such that
$$d^-(u)<\dfrac{|D_1|}{2}+\dfrac{|D_1|\cdot|D_4|}{|D_1|-1}+\dfrac{2a(M)|D_1|}{45\cdot (|D_1|-1)},$$
$$d(v)<|D_1|+|D_3|+|D_4|+ 4a(M)/45,$$
$$d^+(w)<|D_3|/2 + |D_4| + 2a(M)/45.$$
Together with $|D_1|+|D_2|+|D_3|+|D_4|\leq n$, it yields that
\begin{equation}\label{9}
\begin{aligned}
\dfrac{3n}{2}&\leq d^-(u) + d^+(w) + d(v) \\
&<\dfrac{3}{2}(n-|D_2| + |D_4|) + 2|D_4|+ \dfrac{|D_1|\cdot|D_4|}{|D_1|-1} + \dfrac{6a(M)}{45}+\dfrac{2a(M)|D_1|}{45(|D_1|-1)}\\
&<\dfrac{3}{2}(n- |D_2|+ \dfrac{|D_4|\cdot |D_1|}{|D_1|-1})+\dfrac{6a(M)}{45}+\dfrac{2a(M)|D_1|}{45(|D_1|-1)}.
\end{aligned}
\end{equation}

Recall that $|D_2| - |D_4|=s>a(M)$, so $|D_2|- \dfrac{|D_4|\cdot |D_1|}{|D_1|-1}> a(M)$ by simple calculation. Thus (\ref{9}) gives that
$$\dfrac{3}{2}\cdot a(M)< \dfrac{6a(M)}{45}+\dfrac{2a(M)|D_1|}{45(|D_1|-1)}.$$
This is impossible, so $a(M)=0$. Further, (\ref{9}) also implies that $|D_2|- \dfrac{|D_4|\cdot |D_1|}{|D_1|-1}<0$. Namely, $|D_2|-|D_4|=0$. It contradicts the assumption, and thus the proof is completed.
\end{proof}
\end{proof}

%Hence, there must be $|D_2|=|D_4|$ in $H^\prime$. It contradicts Claim \ref{claim1}.

\textbf{Proof of Theorem \ref{cor2}.}	
Suppose $D$ is an oriented graph on $n\geq n_0$ vertices with $\delta^0(D)\geq 3n/8+2$. And assume that we need find a Hamiltonian path from $u$ to $v$. If $N_D^+(u)\cap N_D^-(v)$ is not an empty set, divide $N_D^+(u)\cap N_D^-(v)$ equally into two parts $N_u,N_v$. Next, an auxiliary oriented graph $H$ is obtained by removing the vertices $u,v$ from $D$, then add a vertex $w$ with $N_H^+(w)= N_D^+(u)\setminus  N_v$ and $ N_H^-(w)= N_D^-(v)\setminus N_u$. Then semidegree of $w$ in $H$ is at least $3n/16$ and $\delta^0(H-\{w\})\geq 3n/8$, which is clear from $\delta^0(D)\geq 3n/8+2$. Thus $H$ contains a Hamiltonian cycle since Lemma \ref{main2}. Restore the vertex $w$ into the vertices $u,v$. It means that there is a Hamiltonian path from $u$ to $v$ in $D$. The theorem is established by the arbitrariness of $u$ and $v$.
\hfill $\Box$

\section{Proof of Theorem \ref{cor3}.}
Before proving Theorem \ref{cor3}, we need to introduce more concepts. Denote the \emph{distance} of two vertices $x,y$ to be the length of a shortest path from $x$ to $y$. In 2008, Lichiardopol \cite{Lichiardopol(2008)} showed that the diameter of an oriented graph $D$ of order $n$ with $\delta^0(D)\geq n/3$ is at most 5.

\textbf{Proof of Theorem \ref{cor3}.}	
%We proof Corollary \ref{cor3} by induction on $k$. There is nothing to prove for $k=1,2$. Assume the conclusion holds for $k-1$, we will prove it for $k$.
($i$) Assume the sequence $s_1,\ldots,s_k$ of distinct vertices of $D$ is the order which we need to encounter by a Hamiltonian cycle. We proof Theorem \ref{cor3}$(i)$ by finding a path $P$ from $s_1$ to $s_k$ whose length is at most $4k-3$. For each $j\in [k]$, there exists a path $P_i$ form $s_i$ to $s_{i+1}$ whose length is at most $5$ and avoids the vertices $\{s_1,s_2,\ldots,s_{i-1},s_{i+2},\ldots,s_k\}\cup \bigcup_{j=1}^{i-1}P_j$ since $\delta^0(D-\bigcup^k_{j=1}P_j)\geq 3n/8+5k/2-2-(4k-3)> n/3.$ Therefore, there is a path $P$ from $s_1$ to $s_k$ and encounters each $s_i$ in the order. Apparently, the length of $P$ is at most $4k-3$. Let $H$ be the oriented graph that remove the vertices $P\setminus \{s_1,s_k\}$ from $D$. Owing to $\delta^0(D-P\setminus \{s_1,s_k\})\geq 3n/8+5k/2-2-(4k-3)> 3(n-4k-3)/8+2$ and Theorem \ref{cor2}, there is a Hamiltonian path from $s_k$ to $s_1$ in $H$. This implies that $D$ is $k$-ordered Hamiltonian.

%We proof Corollary \ref{cor3} by induction on $k$. There is nothing to prove for $k=1,2$. Assume the conclusion holds for $k-1$, we will prove it for $k$.
($ii$) Assume $X=\{x_1,x_2,\ldots,x_{k}\}$, $Y=\{y_1,y_2,\ldots,y_{k}\}$ be $2k$ vertices to be linked. Since $\delta^0(D-\bigcup^k_{j=1}P_j)\geq 3n/8+7k/2-1-5k> n/3$, there exists a path $P_i$ from $x_i$ to $y_{i}$ of length at most $5$ and avoiding the vertices $(X\cup Y\cup \bigcup_{j=1}^{i-1}P_j)\setminus\{x_i,y_i\}$, for each $i\in [k-1]$. This yields that the semidegree of $D-\bigcup_{j=1}^{k-1}P_j$ is at least $3n/8+7k/2-1-5(k-1)\geq 3(n-5k+5)/8+2$. Together with Theorem \ref{cor2}, we get that $D-\bigcup_{j=1}^{k-1}P_j$ contains a Hamiltonian path from $x_k$ to $y_k$. Hence $D$ is spanning $k$-linked.
\hfill $\Box$

\section{Remark}
In 1993, H\"{a}ggkvist \cite{Haggkvist} also made the following conjecture. Here, define $\delta^*(D)= \delta(D) + \delta^+(D) + \delta^-(D)$.

\begin{conjecture}\label{conj1}
Every oriented graph $D$ with $\delta^*(D) > (3n - 3)/2$ is Hamiltonian.
\end{conjecture}

In \cite{Kelly(2008)}, this conjecture was verified approximately; that is, if $\delta^*(D)\geq (3/2 + o(1))n$, then the oriented graph $D$ is Hamiltonian. In addition, notice that in Theorem \ref{main1}, we only get any cycle-factor with constant cycles, which implies that there must exist a long cycle whose length at least $n/t$ in any cycle-factor. Sudakov \cite{Sudakov(2009)} described an infinite class of tournaments $T$ of order $n$ with $\delta^0(T)\geq (n-1)/2 - 1$ which are not have a perfect packing of cyclic triangles. Naturally, we put forward the following question.
%not in any hamiltonian cycle.A example which showed that there is a tournament $T$ of order $n$ with $\delta^0(T)\geq (n-1)/2 - 1$ does not have a perfect packing of cyclic triangles is given in \cite{Sudakov(2009)}.

\begin{problem}
Whether there is any cycle-factor with $t$ cycles in an oriented graph $D$ on $n$ vertices with $\delta^0(D)\geq (3n-4)/8$ (resp. $\delta^*(D) > (3n - 3)/2$), for any natural number $t<n/3$?
\end{problem}

Finally, we can also consider the linkage problem in oriented graphs. Indeed, it is easy to check that every oriented graph $D$ on $n$ vertices with $\delta^0(D)\geq n/3+5k-5$ is $k$-linked. On the other hand we can construct an oriented graph $D$ on $n$ vertices with $\delta^0(D)\geq n/4+3k/2-5/2$, but not $k$-linked.

Let $D$ be an oriented graph on vertex set $V =B\cup C\cup X\cup Y$, where $X = \{x_1,\ldots,x_k\}$, $Y = \{y_1,\ldots,y_k\}$ and $|B| = |C | = (n-2k)/2$ and whose arcs are listed in the following ways.\\
(1) The arcs in both $X$ and $Y$ are oriented arbitrarily.\\
(2) The arcs in both $B$ and $C$ are oriented so that $D[B]$ and $D[C]$ form a regular tournament.\\
(3) All arcs are oriented from $\{x_k,y_k\}\cup C$ to $B$, from $B$ to $(X\cup Y)\setminus \{x_k,y_k\}$, from $ C$ to $\{x_k,y_k\}$ and from $(X\cup Y)\setminus \{x_k,y_k\}$ to $C$.\\
(4) All arcs are oriented from $X$ to $Y$ except for arcs $y_ix_i$ for each $i\in [k]$.\\
\begin{figure}[H]
\centering    %居中       %子图居中
   \includegraphics[scale=1.1]{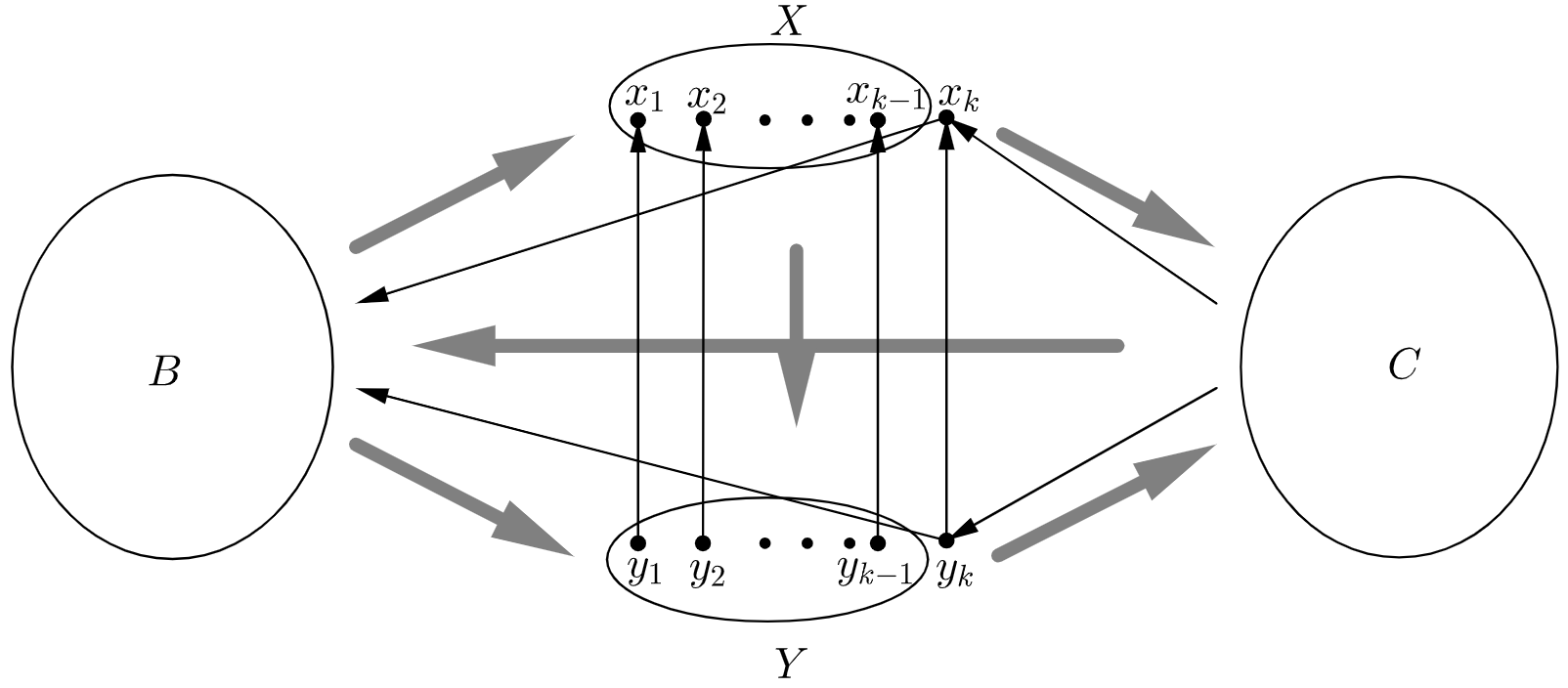}   %以pic.jpg的0.5倍大小输出
\caption{An oriented graph $D$ with $\delta^0(D)\geq n/4+3k/2-5/2$ that is not $k$-linked.}
\label{tu2}
\end{figure}

First, we shall show that $\delta^0(D)\geq n/4+3k/2-5/2$. It is easy to check that every vertex in $\{x_1,\ldots,x_{k-1}, y_1,\ldots,y_{k-1}\}$ has out-neighbourhood $C$ and in-neighbourhood $B$. This gives that its semidegree is at least $ (n-2k)/2$. Similarly, the semidegree of $x_k$ (resp. $y_k$) is at least $ (n-2k)/2$.
Therefore, it remains to show that every vertex in $B\cup C$ has semidegree at least $n/4+3k/2-5/2$. By the construction of (2)-(3), for   every vertex $x\in B\cup C$, there is
$$d^{\sigma}(x)\geq \dfrac{(n-2k)/2-1}{2}+2(k-1)\geq n/4+3k/2-5/2, \sigma\in \{+,-\}.$$
But it is not difficult to see that $D$ does not contain an ($x_k$, $y_k$)-path which avoids $X\cup Y$.
%Recall that $D[B]$, $D[C]$ form a regular tournament and $(X\cup Y)\setminus \{x_k,y_k\}$

Therefore, we provide the following problem.
\begin{problem}
Whether there is a constant $c$ such that every oriented graph $D$ on $n$ vertices with $\delta^0(D)\geq n/4+ck$ is $k$-linked?
\end{problem}

\end{document}